\documentclass[11pt]{article}
\usepackage{mathrsfs}
\usepackage{amsmath}
\usepackage{amssymb}
\usepackage[amsmath,thmmarks,hyperref]{ntheorem}
\usepackage[all]{xy}
\usepackage{stmaryrd}
\usepackage{float}
\usepackage[bookmarksnumbered]{hyperref}
\usepackage{enumitem}
\usepackage{eucal}
\usepackage{comparison}
\usepackage{todonotes}
\usepackage{rotating}
\usepackage{appendix}
\usepackage{parskip}
\usepackage[textwidth=15.5cm, textheight=22cm]{geometry}
\usepackage{pdflscape}
\usepackage{mathabx}
\usepackage{float}
\newcommand{\R}{\mathbb{R}}

\newcommand{\ra}{\rightarrow}
\newcommand{\lra}{\longrightarrow}

\newcommand{\rcong}{ \begin{sideways}$\cong$\end{sideways}}

\newcommand{\join}{\Asterisk}
\newcommand{\smsh}{\wedge}
\newcommand{\ox}{\otimes}
\newcommand{\x}{\times}

\newcommand{\pt}
{\bullet}

\newcommand{\T}{\mathrm{T}}
\renewcommand{\P}{\mathrm{P}}

\DeclareMathOperator{\nhomog}{n--homog}
\DeclareMathOperator{\nexs}{n--exs}
\DeclareMathOperator{\npoly}{n--poly}

\DeclareMathOperator{\res}{Res}
\DeclareMathOperator{\ind}{Ind}
\DeclareMathOperator{\Sp}{Sp}

\DeclareMathOperator{\topspace}{Top}
\DeclareMathOperator{\diff}{diff}

\DeclareMathOperator{\skel}{sk}
\DeclareMathOperator{\cref}{cr}

\DeclareMathOperator{\nat}{Nat}
\DeclareMathOperator{\mapdiag}{map--diag}

\newcommand{\devcat}{\Sigma_n \ltimes (\wcal_n \Top)}
\newcommand{\orthdevcat}{O(n) \ltimes (\jcal_n \Top)}
\newcommand{\interorthdevcat}{\Sigma_n \ltimes (i^* \jcal_n \Top)}

\newcommand{\lca}{\circlearrowleft}

\newcommand{\bigsmashprod}[2]{\underset{#1}{\overset{#2}{\bigwedge}}}
\newcommand{\Top}{{\topspace}}

\newcommand{\wt}{\mathbb{T}}

\title{Comparing the orthogonal and homotopy functor calculi}
\author{David Barnes \and Rosona Eldred}
\date{}

\begin{document}
\maketitle
\begin{abstract}
\noindent 
Goodwillie's homotopy functor calculus constructs a
Taylor tower of approximations to $F$, often a functor from spaces to spaces.
Weiss's orthogonal calculus provides a Taylor tower for
functors from vector spaces to spaces. In particular, there is a Weiss tower associated to the
functor $V \mapsto F (S^V)$, where $S^V$ is the one-point compactification of $V$.

In this paper, we give a comparison of these two towers and show
that when $F$ is analytic the towers agree up to weak equivalence.
We include two main applications, one of which gives as a corollary the convergence of the Weiss Taylor tower of $BO$. We also lift the homotopy level tower comparison to a commutative diagram
of Quillen functors, relating model categories for Goodwillie calculus and model categories for
the orthogonal calculus.
\end{abstract}

{\small \tableofcontents}

\section{Introduction}

Goodwillie's calculus of homotopy functors, developed originally in  \cite{gw90, gw91, goodcalc3},  is a method of studying
equivalence-preserving functors,  motivated by applications to Waldhausen's algebraic $K$-theory of a space. A family of related theories grew out of this work;  our focus is on the homotopy functor calculus and the orthogonal calculus of Weiss \cite{weiss95}, the latter of which  was developed to study functors from real inner-product spaces to topological spaces, such as $BO(V)$ and $TOP(V)$.

The model categorical foundations for the homotopy functor calculus and the orthogonal calculus may be found in 
Biedermann-Chorny-R\"ondigs \cite{BCR07}, Biedermann-R\"ondigs \cite{BRgoodwillie} and Barnes-Oman\cite{barnesoman13} and, most recently, the prequel to this paper, Barnes-Eldred \cite{barneseldred}. 
In \cite{barneseldred}, we re-work the classification results of Goodwillie so
as to resemble that of the orthogonal calculus. 
In this paper, we use this
similarity to give a formal comparison between the tower arising from orthogonal calculus and that arising from Goodwillie's calculus of homotopy functors. 
Allusions to such a comparison have existed as folk results for some time,
(for example, it is central to \cite{ADL}).
In more general terms, this paper will make it easier to use the two forms of calculus together and transfer calculations between them. For example, the convergence results of Section \ref{sec:orthogweaklytic} follow from transferring statements about the Goodwillie tower of a functor to the corresponding Weiss tower.

%


We start with a background section, Section \ref{sec:background}, with necessary definitions and results from the calculi. Section \ref{sec:equivtowers} contains a functor level comparison of the theories, with  main results as follows. Given a functor, $F$,
from based topological spaces to based topological spaces we can consider 
the functor $V \mapsto F (S^V)$, which we call the restriction of $F$. 
Proposition \ref{prop:restrictionhomog} shows that the 
restriction of an $n$-homogeneous functor 
(in the sense of Goodwillie) gives an $n$-homogeneous
(in the sense Weiss).
Similarly Proposition \ref{prop:nexc-north} shows that 
restriction sends $n$-excisive functors to $n$-polynomial functors.
We can then consider the restriction of the Goodwillie tower of $F$
and the Weiss tower associated to the
functor $V \mapsto F (S^V)$. 
Theorem \ref{thm:tower-eq} shows that when $F$ is analytic, 
these two towers agree.

From these results, we obtain two applications. Firstly, we show in 
Section \ref{sec:orthogweaklytic} that the
Weiss tower of the functor $V \mapsto BO(V)$ converges to $BO(V)$
when $V$ has dimension at least 2 (Corollary \ref{cor:BOV}). This convergence was claimed without proof in Arone \cite[p.13]{aroneweiss}.  This result follows from the more general result we establish (Theorem \ref{thm:weissanalytic}) that if the 1st (unstable) derivative of a functor has a kind of analyticity, so does the functor.
Secondly, we lift the comparisons of the two forms of calculus to 
a commutative diagram of model categories and Quillen 
pairs, see  Figure \ref{fig:big} of Section \ref{sec:diagram}. 
This diagram compares
the model structures of $n$-excisive functors, $n$-polynomial functors
$n$-homogeneous functors (of both kinds) and the 
categories of spectra that classify
homogeneous functors.

\paragraph{Acknowledgements} 
The authors would like to thank  
Michael Weiss for motivating discussions on Section \ref{sec:orthogweaklytic} 
and Greg Arone and Tom Goodwillie for numerous helpful conversations.

\section{Background}\label{sec:background}
In this section we introduce the orthogonal calculus
of Weiss and the homotopy functor calculus of Goodwillie. 
For the sake of the comparison, we use compatible model category
versions of these calculi, namely work of Barnes-Oman \cite{barnesoman13} for orthogonal
calculus and Barnes-Eldred \cite{barneseldred} for the homotopy functor calculus.

We are only interested in the based versions of these
two forms of calculus; all spaces are
based. We let $\Top$ denote the category of
based (compactly generated, weak Hausdorff) topological spaces.

\subsection{Homotopy Calculus}
In either of our two settings, the input $F$ will be a functor from some small category
to based spaces. The output will be a tower of functors
of the same type as $F$.
That is, for each $n \geqslant 0$ there will be a
fibration sequence $D_n F \to \mathbb{P}_n F \to \mathbb{P}_{n-1} F$, which can be arranged as below.
\[
\xymatrix@C+1cm@R-0.5cm{
& \vdots \ar[d] \\
& \mathbb{P}_3 F \ar[d] & D_3 F \ar[l] \\
& \mathbb{P}_2 F \ar[d] & D_2 F \ar[l] \\
& \mathbb{P}_1 F \ar[d] & D_1 F \ar[l] \\
F \ar[r] \ar[ur] \ar[uur] \ar[uuur]
&  \mathbb{P}_0 F. \\
}
\]
The functors $\mathbb{P}_nF$ (which for the hofunctor calculus are called $P_nF$ and for the orthogonal calculus, $\wt_nF$) have a kind of $n$-polynomial property, and for nice functors, the inverse limit of the tower, denoted $\mathbb{P}_\infty F$, is equivalent to $F$. We think of each $\mathbb{P}_n F$ as an `$n^{th}$-approximation' to $F$, moreover $\mathbb{P}_nF$ 
can be recognised 
as a fibrant replacement of $F$ in some suitable model category. The layers of the tower, $D_nF$, are then analogous to purely-$n$-polynomial functors -- called $n$-homogeneous.

In each case there will be a classification theorem
giving the homogeneous functors in terms of some form of spectra.
Furthermore the spectrum corresponding to $D_n F$ 
can be calculated without first constructing $\mathbb{P}_n F$ and $\mathbb{P}_{n-1} F$.

For each theory, we will give a definition of the relevant notion of $n$-polynomialness,
a construction of the approximations $\mathbb{P}_n F$
and a direct construction of spectrum 
corresponding to the functors $D_n F$
and a classification theorem.

\subsection{The categories of functors}
We introduce the necessary basic categories of functors studied in the homotopy functor calculus and the orthogonal calculus. 

For $\gcal$ a topological group we let
$\gcal \lca \Top$ denote the category of based topological spaces with $\gcal$-action (that fixes the basepoint) and equivariant maps. 

\begin{definition}
  Let $\wcal$ be the category of finite based CW-complexes, which is enriched
  over the category of based topological spaces.
  Let $\ical$ be the category of finite dimensional real inner product
  spaces with morphisms the isometries. Let $\jcal_0$ be the category with the same objects
  as $\ical$, but with morphism spaces given by $\jcal_0(U,V)= \ical(U,V)_+$.
  The category $\jcal_0$ is also enriched over $\Top$.
\end{definition}

We note that $\jcal_0(V,V)=O(V)_+$, and that $\jcal_0(U,V)= \ast$ whenever the dimension
of $U$ is larger than the dimension of $V$.

\begin{definition}
  Let $\wcal \Top$ and $\jcal_0 \Top$ denote the categories of enriched functors from
  $\wcal$ and $\jcal_0$ to $\Top$.
\end{definition}

Since both $\jcal_0$ and $\wcal$ are skeletally small, $\wcal \Top$ and $\jcal_0 \Top$ have all small
limits and colimits, constructed objectwise.

\begin{lemma}
All functors $F$ in $\wcal \Top$ are reduced homotopy functors. That is, $F$ preserves weak equivalences and $F(\ast)$ is equal to $\ast$, the one point space. 
\end{lemma}
\begin{proof}
Given any $A$ and $B$ in $\wcal$, the composite
$A \smashprod B \cong A \smashprod \wcal(S^0,B)  \to \wcal(A, A \smashprod B)$

induces a map $F(A) \smashprod B \to F(A \smashprod B)$. 
If $f,g \co A \to A'$ are maps in $\wcal$ which are homotopic, 
then we can construct a homotopy between $F(f)$ and $F(g)$ 
using the above map in the case $B =[0,1]_+$. 

This lets us conclude that $F(\ast) \simeq \ast$.  
We proceed to show equality: let $0_X \co X \to X$ denote the
constant map which sends $X$ to the basepoint of $X$. 
The functor $F$ is enriched over based spaces, so $F(0_X) = 0_{F(X)}$. 
Set $X=\ast$, then $\id_{F(\ast)}= F(\id_\ast)$, which is equal to 
$F(0_\ast) = 0_{F(\ast)}$.  Hence $F(\ast)$ is $\ast$. 
\end{proof}

Since every weak homotopy equivalence in $\wcal$ is a homotopy equivalence, 
we see that every object of $\wcal \Top$ preserves weak equivalences. 

\begin{definition}
For $A$ and $B$ in $\wcal$, the map
$F(A) \smashprod B \to F(A \smashprod B)$
constructed in the previous proof 
is called the \textbf{assembly map}.
\end{definition}



\subsection{The approximations}
We specify the $n$-polynomial properties that occur in the functor calculus and the orthogonal calculus.

  Let $\pcal (\underline{n})$ denote the poset of subsets of
  $\{1,\ldots, n\}$;  $\pcal_0(\underline{n})$ denotes the \textit{non-empty} subsets and $\pcal^1 (\underline{n})$ is all but the final set.
  
\begin{definition}
  A functor $F \in \wcal \Top$ is said to be \textbf{$n$-excisive} if
  for any $X$ the map (induced by the inclusion of $\emptyset$ in the poset of subsets of $\underline{n}$)
  \[F(X) \lra \holim_{S \in \mathscr{P}_0(\underline{n+1})} F(S \join X) \]
  is a weak homotopy equivalence for each $X$ in $\wcal$ ($\join$ denotes the topological join).
\end{definition}

An $n$-cube $\xcal$ in a category $\ccal$ is a functor $\xcal: \pcal (\underline{n})\lra \ccal$.
\begin{definition}
We say an $n$-cube $\xcal$ is \textbf{cartesian} if $\xcal (\emptyset) \overset{\simeq}{\lra} \holim_{\pcal_0 (\underline{n})} \xcal$. Dually, $\xcal$ is \textbf{cocartesian} if $\hocolim_{\pcal^1 (\underline{n})} \overset{\simeq}{\lra}\xcal (\underline{n})$. An $n$-cube is \textbf{strongly cocartesian} if every sub-2-cube is cocartesian. 2-cartesian diagrams are homotopy pullback squares, 2-cocartesian are homotopy pullback squares. 
\end{definition}

The above definition of $n$-excisive is equivalent to the more standard statement
about a functor taking strongly cocartesian $(n+1)$-cubes to
cartesian $(n+1)$-cubes. A functor is 1-excisive if and
only if it takes homotopy pushouts to homotopy pullbacks.

\begin{definition}
  For $E \in \jcal_0 \Top$ define
\[
\tau_n E(V) = \underset{0 \neq U \subset \rr^{n+1}}{\holim}
E(U \oplus V)
\]
We say that $E$ is \textbf{$n$-polynomial} if
\[
\rho_E^n (V) \co  E(V) \lra \tau_n E(V)
\]
is a weak homotopy equivalence for each inner product space $V$.
\end{definition}

\begin{definition}
  We define $\wt_n E = \hocolim \tau_n^k E$ for $E \in \jcal_0 \Top$.
  Similarly, we define $P_n F = \hocolim \T_n^k F$ for $F \in \wcal \Top$.
  We call $\wt_n$ the \textbf{$n$-polynomial approximation functor}
  and we call $P_n$ the \textbf{$n$-excisive approximation functor}.
\end{definition}

The justification for calling these functors approximations comes from the following result, 
see \cite[Theorem 1.8]{goodcalc3} and \cite[Theorem 6.3]{weiss95}.

\begin{proposition}
If $F'$ is $n$-excisive, then any map $F \to F'$ factors (up to homotopy)
over the natural map $F \to P_n F$.
If $E'$ is $n$-polynomial, then any map $E \to E'$ factors (up to homotopy)
over the natural map $E \to \wt_n E$.
\end{proposition}

The following result is \cite[Proposition 3.2]{gw91} and \cite[Proposition 5.4]{weiss95}.

\begin{proposition}\label{prop:n-is-n+1}
  Every $(n-1)$-polynomial functor is $n$-polynomial.
  Every $(n-1)$-excisive functor is $n$-excisive.
\end{proposition}

Combining these two results gives canonical (up to homotopy)
maps $P_n F \to P_{n-1} F$ and $\wt_n E \to \wt_{n-1} E$
for any $F \in \wcal \Top$ and any $E \in \jcal_0 \Top$.

\subsection{The homogeneous functors}

We define $n$-homogeneous functors in each setting and give the classification results of Goodwillie and Weiss.

\begin{definition}
  For $E \in \jcal_0 \Top$, we define $D_n^W E$ to be the homotopy fibre of
  $\wt_n E \to \wt_{n-1} E$ (the $W$ stands for Weiss).
  For $F \in \wcal \Top$, we define $D_n^G F$ to be the homotopy fibre of
  $P_n F \to P_{n-1} F$ (the $G$ stands for Goodwillie).
\end{definition}

\begin{definition}
  We say that a functor $F \in \wcal \Top$ is \textbf{$n$-homogeneous} if it is $n$-excisive and 
  $P_{n-1}F$ is objectwise contractible. 
  We say that a functor $E \in \jcal_0 \Top$ is \textbf{$n$-homogeneous} if it is $n$-polynomial and
  $\wt_{n-1}E$ is objectwise contractible. 
\end{definition}
Since $P_{n-1}$ and $\wt_{n-1}$ commute with homotopy fibres, 
the functors $D^G_nF$ and $D^W_n E$ are $n$-homogeneous.

The next result classifies the homogeneous functors (in either setting) in terms of 
spectra with group actions. For details, see \cite[Sections 2-5]{goodcalc3} and \cite[Theorem 7.3]{weiss95}.

\begin{theorem}\label{thm:classifications}
  The full subcategory of $n$-homogenous functors inside $\ho( \wcal \Top )$ is 
  equivalent to the homotopy category of spectra with $\Sigma_n$-action. 
  Given a spectrum $\Theta_F$ with $\Sigma_n$-action the functor below 
  is an $n$-homogeneous functor of $\wcal \Top$.
  \[
  X \mapsto \Omega^\infty \big( (\Theta_F \smashprod X^{\smashprod n})/h \Sigma_n \big)
  \]
  The full subcategory of $n$-homogenous functors inside $\ho( \jcal_0 \Top )$ is
  equivalent to the homotopy category of spectra with $O(n)$-action. 
  Given a spectrum $\Psi_E$ with $O(n)$-action the functor below
  is an $n$-homogeneous functor of $\jcal_0 \Top$.
  \[
  V \mapsto \Omega^\infty \big( (\Psi_E \smashprod S^{\rr^n \otimes V} )/h O(n) \big)
  \]
\end{theorem}

We now elaborate on how one obtains the spectra $\Theta_F$ and $\Psi_E$.
We begin with the orthogonal calculus setting. Recall that we denote by $\ical$ the category of finite dimensional real inner
product spaces with morphisms the isometries. %
Define a vector bundle over $\ical (U,V)$, for $U,V \in \ical$.
\[
\gamma_n (U,V) = \{ (f,x) | f: U \to V,   x \in \rr^n \otimes (V-f(U))\}
\]
The total space has a natural action of $O(n)$ due to the $\rr^n$ factor.
We then let $\jcal_n (U,V):= T\gamma_n (U,V)$, the associated Thom space.  
This is the cofibre in the sequence:
\[
\begin{array}{cccccc}
S \gamma_n (U,V)_+ & \lra & D \gamma_n (U,V)_+ & \lra & T\gamma_n (U,V) \\
\rcong & &\rcong \\
\{(f,x) \mid \ ||x|| =1\} && \{(f,x) \mid \ ||x|| \leq 1\}\\
\end{array}
\]
Recall that $T( \R^n \ra \ast) = S^n$ and $T(X=X) = X_+$, for compact $X$.
In particular, if we let $n=0$, we see that $\jcal_0(U,V)= \ical(U,V)_+$ as already defined.

There is a natural composition when looking at the bundles
\[
\begin{array}{cccccccc}
\gamma_n (V,X) &\x& \gamma_n (U,V) & \lra & \gamma_n (U,W) \\
(g,y)          &  & (f,x) & \mapsto & (g \circ f, y + (\R^n \ox g)x)\\
\end{array}
\]
where $(\R^n \ox g):\R^n \ox (V-f(U)) \ra \R^n \ox W$. This composition induces unital and associative maps
\[
\jcal_n (V,W) \smsh \jcal_n (U,V) \ra \jcal_n (U,W),
\]
which are $O(n)$-equivariant and functorial in the inputs. 

\begin{definition}\label{def:jcaln}
We define an $O(n)\lca \Top$-enriched category $\jcal_n$, whose objects are 
finite dimensional real inner product spaces and morphism spaces are given by $\jcal_n (U,V)$.

The map $i_n \co \jcal_0(U,V) \to \jcal_n(U,V)$, $f \mapsto (f,0)$ induces a map of enriched categories
$\jcal_0 \to \jcal_n$. 
\end{definition}

The map of enriched categories $i_n \co \jcal_0 \to \jcal_n$ induces a (restriction) functor $\res_0^n:= i_n^*$
from $\jcal_n \Top$ to $\jcal_0 \Top$. This functor has a right adjoint
$\ind_0^n$ by \cite[Proposition 2.1]{weiss95}. In the following we abuse notation
slightly and write $\res_0^n \ind_0^n$ as simply $\ind_0^n$.

\begin{definition}\label{def:nthderiv}
  For $E \in \jcal_0 \Top$, we define the 
  \textbf{unstable $n^{th}$-derivative of $E$}, 
  $\ind_0^n E \in \jcal_0 \Top$ by 
  \[
  \ind_0^n E (V) = \nat_{\jcal_0 \Top} (\jcal_n(V,-), E)
  \]
  Note that we can give $\ind_0^n E (V)$ an action of $O(n)$ by letting $O(n)$
  act on $\jcal_n(V,-)$. 
\end{definition}

In fact, we can give $\ind_0^n E$ much more structure. 
In \cite{barnesoman13}, it was shown that this is an object of the category $\orthdevcat$ of $O(n) \lca \Top$-enriched functors from
$\jcal_n$ to $O(n) \lca \Top$. 

\begin{remark}
As a functor in $\jcal_0 \Top$, we call $\ind_0^n E$ the \textbf{unstable $n^{th}$ derivative} of $E$. As $\orthdevcat$ has a stable model structure, it is logical to call 
$\ind_0^n E \in \orthdevcat$ the (stable) 
\textbf{$n$th derivative} of $F$.
\end{remark}

Furthermore, $\orthdevcat$ has a stable model structure and an adjunction with
spectra with an $O(n)$-action which is a Quillen equivalence. 
We will elaborate on this further in section \ref{sec:diagram}.
For now, we simply give the following (imprecise) theorem.

\begin{theorem}\label{thm:weissclass}
  The $O(n)$-spaces $\ind_0^n E(V)$ for $V \in \jcal_0$ 
  define a spectrum $\Psi_E$ with $O(n)$-action. 
  Futhermore, the $n$-homogeneous functor constructed from $\Psi_E$ by Theorem \ref{thm:classifications}
  is equivalent to $D_n^W E$.
\end{theorem}

Now we turn to the homotopy functor calculus equivalent and consider a more direct way to construct the spectrum corresponding to $D_n^G F$. 

\begin{definition} \label{def:cref}
For $F \in \wcal \Top$ and an $n$-tuple of spaces in $\wcal$, $(X_1, \ldots, X_n)$, 
the \textbf{$n^{th}$--cross effect of $F$
at $(X_1, \ldots, X_n)$} is the space
\[
\cref_n(F)(X_1, \ldots, X_n) = \nat(\bigsmashprod{l=1}{n} \wcal(X_l,-), F)
\]
Pre-composing $\cref_n(F)$ with the diagonal map
$\wcal(X,Y) \to \bigsmashprod{i=1}{n} \wcal(X,Y)$
yields an object $\diff_n(F)$ of $\wcal \Top$.
\end{definition}

Reworking the above, we obtain a form much closer to that of the orthogonal calculus.

\begin{definition}\label{def:wcaln}
The category $\wcal_n$ is a $\Sigma_n \lca \Top$-enriched category with objects the finite dimensional $CW$-complexes and morphism spaces given by
\[
\wcal_n(A,B) = \bigsmashprod{k=1}{n} \wcal(A,B).
\]
The group $\Sigma_n$ acts on this space by permuting factors.
We may then define 
\[
\diff_n(F)(X) = \nat(\wcal_n(X,-), F)
\]
\end{definition}


In fact we can give $\diff_n(F)$ much more structure, 
\cite[Section 6]{barneseldred} shows that it can be considered as 
an object of the category $\devcat$ of $\Sigma_n \lca \Top$-enriched functors from $\wcal_n$ to $\Sigma_n \lca \Top$. 

In \cite{barneseldred} is also shown that (when equipped with a suitable model structure) 
$\devcat$ is Quillen equivalent to spectra with an $\Sigma_n$-action.

\begin{remark}
As a functor in $\wcal \Top$, we call $\diff_n (F)$ the \textbf{unstable $n$th derivative} of $F$. As $\devcat$ has a stable model structure, it is logical to call $\diff_n (F) \in \devcat$ the (stable) \textbf{$n$th derivative} of $F$.
\end{remark}

We will elaborate on this further in section \ref{sec:diagram}.
For now we simply give the following (imprecise) theorem.

\begin{theorem}\label{thm:goodclass}
  The $\Sigma_n$-spaces $\diff_n F(X)$ define a spectrum $\Theta_F$ with $\Sigma_n$-action.
  Furthermore, the $n$-homogeneous functor constructed from $\Theta_F$ by Theorem \ref{thm:classifications}
  is $D_n^G F$.
\end{theorem}



\subsection{Connectivity and analyticity}\label{sec:connanalyt}
Notions of convergence and how well a functor is approximated by its tower of $n$-excisive (respectively, $n$-polynomial) approximations rely on the following definitions of stably $n$-excisive and $\rho$-analyticity, as well as their orthogonal analogs. The results for orthogonal calculus are a re-phrasing and expansion on results found in \cite{weisserrata}.  We will make use of these especially in Section \ref{sec:equivtowers}. 

\label{sec:stabndef}

Roughly speaking, if $F \in \wcal \Top$ is stably $n$-excisive then it takes strong cocartesian $(n+1)$-cubes to \emph{almost} cartesian cubes. 

 \begin{definition}\cite[Defn 4.1]{gw91}\label{def:stabnexc}
$F \in \wcal \Top$ is \textbf{stably $n$-excisive} or satisfies \textbf{stable $n^{th}$ order excision} if the following holds for some numbers $c$ and $\kappa$:
\begin{quote}
\begin{description}
\item[$E_n (c, \kappa):$] If $\mathscr{X}: \mathscr{P} (S) \ra \mathscr{C}$ is any strongly co-Cartesian $(n+1)$-cube such that $\forall s \in S$ the map $\mathscr{X} (\emptyset)\ra \mathscr{X} (s)$ is $k_s$-connected and $k_s \geq \kappa$, then the diagram $F(\mathscr{X})$ is $(-c+ \sum k_s)$-Cartesian.
\end{description}
\end{quote}
 \end{definition}

It is usual to instead consider the following property which is satisfied for some $\rho$ by many important functors. 
\begin{definition}\cite[Defn 4.2]{gw91}\label{def:kanalytic} $F \in \wcal \Top$ is \textbf{$\rho$-analytic} if there is some number $q$ such that $F$ satisfies $E_n(n\rho-q, \rho+1)$ for all $n\geq 1$.
  \end{definition}

\begin{remark}\label{rem:analytic-stab}
By definition, $\rho$-analytic functors are stably $n$-excisive for all $n$. 	
\end{remark}

One of the main consequences of $\rho$-analyticity is the following, which we formally define as it is a relevant notion for orthogonal calculus as well; for functors in $\jcal \Top$, we use dimension instead of connectivity. 

\begin{definition}\label{def:weakanalytic}
We say that $F\in \wcal \Top$ is \textbf{weakly $\rho$-analytic} if for any space of connectivity at least $\rho$, $F(X) \overset{\simeq}{\lra} P_\infty F(X)$. If $F$ is $\rho$-analytic, then $F$ is weakly $\rho$-analytic. 

We say that $E\in \jcal_0 \Top$ is \textbf{weakly $\rho$-analytic} if for any vector space of dimension at least $\rho$, $E(X) \overset{\simeq}{\lra} \wt_\infty E(X)$. 
\end{definition}

 \begin{definition}\cite[Definition 1.2]{goodcalc3} \label{def:ordern}
 A map $u: F\ra G$ in $\wcal \Top$ satisfies $O_n ( c, \kappa)$
 if $\forall k \geq \kappa \; \forall X \in C$ with $X$ $k$-connected, $u_X:FX \ra GX$ is $(-c+(n+1)k)$-connected.

 We say that $F$ and $G$ \textbf{agree to order $n$} (via $u$) if $u$ satisfies $O_n(c,\kappa)$ for \textbf{some} $c, \kappa$. We also call
 $u$ an \textbf{agreement to order $n$}.
 \end{definition}

The orthogonal analog translates connectivity into dimension: 
 \begin{definition}
 \label{def:orthordn}
Let $p: F \ra G$ be a morphism in $\jcal_0\Top$. Suppose there exists a $b$ such that $p:FW \ra GW$ is $(-b + (n+1)\textrm{dim}(W))$-connected for all $W \in \mathcal{J}_0$.  Then we say that \textbf{$F$ and $G$ agree up to orth-order $n$} via $p$ or $p$ is an \textbf{orth-order-$n$ agreement} between $F$ and $G$.
 \end{definition}

 \begin{prop}\cite[Proposition 1.5]{goodcalc3}
If $F$ is \textbf{stably $n$-excisive}, then
\begin{enumerate}[label=\arabic*]
\item $P_n F$ is $n$-excisive and
\item $F$ agrees with $P_n F$ \textbf{to order $n$} (via $p_n: F \ra P_n F$)
\end{enumerate}
 \end{prop}

\begin{lemma}\label{lem:ordnorthordn}
If $F$ is stably $n$-excisive, then $p_{n, S^V}: F\circ S^V \ra P_n F\circ S^V$ is an orth-order-$n$ agreement.
\end{lemma}

\begin{proof}
The connectivity of $S^V$ is dim$(V)-1$ and 
$F$ is stably $n$-excisive, so \cite[Proposition1.5]{goodcalc3}
implies that $p_{n, S^V}: F\circ S^V \ra P_n F\circ S^V$ has connectivity
\[
(-c+(n+1)\text{conn}(S^V))) = (-c-(n+1) +(n+1)\text{dim}(V))
\]
 for some $c$. Taking $-b = -c+(n+1)$, we conclude that $p_{n, S^V}: F\circ S^V \ra P_n F\circ S^V$ is an orth-order-$n$ agreement.
\end{proof}

\begin{prop}\cite[Proposition 1.6]{goodcalc3}\label{prop:gc3-1.6}
 Let $u: F \ra G$ be a map in $\wcal \Top$. If $F$ and $G$ \textbf{agree to order $n$} (via $u$) then the induced map $P_nu: P_n F \ra P_n G$ is an objectwise weak equivalence.
 \end{prop}

 \label{sec:orthstabndef}

Proposition  \ref{prop:gc3-1.6} has the following orthogonal calculus analogue. The proof of this result is part of the proof of \cite[Lemma e.7]{weisserrata}.

\begin{prop} \label{lem:ordnWneq}
Let $F,G$ be functors in $\jcal_0 \Top$. If $p: F\ra G$ is an \textbf{orth-order-$n$ agreement}, then $\wt_n(p): \wt_n F \ra \wt_n G$ is an objectwise weak equivalence.
\end{prop}


\begin{remark}\label{rem:decrease}
Notice that since the connectivity of $S^V$ is dim$(V)-1$, if $F \in \wcal \Top$ is $\rho$-analytic, we expect $S^\ast F$ to be (weakly) $(\rho+1)$-analytic. See Corollary \ref{cor:GSVanlytic} and Example \ref{ex:SVanalytic}.
\end{remark}

\section{Equivalence of the Weiss and Goodwillie tower under \texorpdfstring{$S^\ast$}{S*}}\label{sec:equivtowers}


We can view $\jcal_0$ via $V\mapsto S^V$ as a subcategory of $\wcal$
(see \cite[Remark 4.7]{mmss01}).  Let $F$ be a homotopy functor, i.e. $F \in \wcal \Top$. We then call $V \mapsto F\circ S^V=(S^\ast F)(V)$ its \textit{restriction} to the image of $\jcal_0$ in $\wcal$.
In this section we will show that the restriction of the Goodwillie tower of a $\rho$-analytic functor is weakly equivalent to the Weiss tower of its restriction.
Definitions of the terms $\rho$-analytic and stably $n$-excisive may be found in Section \ref{sec:stabndef}. Heuristically, these are both connectivity assumptions on what a functor does to strongly cocartesian cubes.






\subsection{Homogeneous functors and restriction}\label{sec:homog}
We begin by considering $n$-homogeneous functors and use 
the classification results of Goodwillie and Weiss to conclude that the restriction
of an $n$-homogeneous functor in $\wcal \Top$ is $n$-homogeneous in $\jcal_0 \Top$. 

\begin{proposition}\label{prop:restrictionhomog}
If a functor $F$ in $\wcal \Top$ is $n$-homogeneous in the sense of Goodwillie,
then $S^* F \in \jcal_0 \Top$ is $n$-homogeneous in the sense of
Weiss. 
\end{proposition}
\begin{proof}
Let $F$ be an $n$-homogeneous functor 
in $\wcal \Top$.
Then, by Theorem \ref{thm:classifications} and Theorem \ref{thm:goodclass},
from $F$ we can obtain $\Theta_F$, a spectrum with an action of $\Sigma_n$ such that
\[
F(A) \simeq
\Omega^\infty (A^{\smashprod n} \smashprod \Theta_F)/{h \Sigma_n}
\quad, \forall A \in \wcal
\]
Using the (derived) change of groups functor, we can construct a
spectrum with an action of $O(n)$ from $\Theta_F$:
\[
O(n)_+ \smashprod_{\Sigma_n}^L \Theta_F = 
O(n)_+ \smashprod_{\Sigma_n} ((E \Sigma_n)_+ \smashprod \Theta_F)
\]
Using Theorem \ref{thm:classifications}, we obtain an $n$-homogeneous functor
$F'$ in $\jcal_0 \Top$
\[
F'(V) = \Omega^\infty (S^{nV} \smashprod O(n)_+ \smashprod_{\Sigma_n}
((E \Sigma_n)_+ \smashprod \Theta_F))/h O(n).
\]
If $X$ is a space with $\Sigma_n$-action and $Y$ is a space
with $O(n)$-action, we have an isomorphism of $O(n)$-spaces
$(O(n)_+ \smashprod_{\Sigma_n} X) \smashprod Y \to
O(n)_+ \smashprod_{\Sigma_n} (X \smashprod i^* Y)$
given by $[g,x,y] \mapsto [g,x,g^{-1} y]$.
This isomorphism extends to the level of spectra, to give an isomorphism 
\[
F'(V) \cong \Omega^\infty (O(n)_+ \smashprod_{\Sigma_n} 
(i^* S^{nV} \smashprod ((E \Sigma_n)_+ \smashprod \Theta_F))/h O(n).
\]
The action of $O(n)$ on $(O(n)_+ \smashprod_{\Sigma_n} 
(i^* S^{nV} \smashprod ((E \Sigma_n)_+ \smashprod \Theta_F))$
is free: the only fixed point of each level of the spectrum is the basepoint, 
since $(E \Sigma_n)_+$ is a free $\Sigma_n$-space. 
Thus, taking $O(n)$-homotopy orbits is the same as taking strict orbits.
We hence have a series of isomorphisms as below.
\[
\begin{array}{rcl}
F'(V) & \cong & \Omega^\infty (O(n)_+ \smashprod_{\Sigma_n} 
(i^* S^{nV} \smashprod ((E \Sigma_n)_+ \smashprod \Theta_F)))/ O(n) \\
& \cong &
\Omega^\infty (i^* S^{nV} \smashprod ((E \Sigma_n)_+ \smashprod \Theta_F))/ \Sigma_n \\
& \cong &
\Omega^\infty (i^* S^{nV} \smashprod \Theta_F)/ h \Sigma_n \\
& \simeq & F(S^V)
\end{array}
\]
We have shown that $S^* F$ is objectwise weakly equivalent to $F'$ and hence it is an $n$-homogeneous functor in the sense of Weiss. 
\end{proof}





Let $\ho (\wcal \Top)$ be the 
homotopy category of $\wcal \Top$, where a weak equivalence
is a objectwise weak equivalence (see Proposition \ref{prop:levelwisemodel}).
We then define $\nhomog \ho (\wcal \Top)$ to be the full subcategory of 
$\ho \wcal \Top$ with objects the $n$-homogeneous functors; we define $\nhomog  \ho (\jcal_0 \Top)$ analogously. 
For $\gcal$ a group, let $\ho (\gcal \lca \Sp)$ be the homotopy category of spectra
with an action of $\gcal$.

Then the previous result shows that the derived functor $RS^*$ of $S^*$ 
is well defined on the categories of $n$-homogeneous functors.
Futhermore, it shows that Figure \ref{fig:completehomotopylevel} is commutative, 
where the vertical arrows are from the classification statements of Weiss and Goodwillie.

\begin{figure}[H]
\[
\xymatrix@C+2cm{
\ho (\Sigma_n \lca \Sp)
\ar[d]
\ar[r]^{O(n)_+ \smashprod_{\Sigma_n}^L -}
&
\ho (O(n) \lca  \Sp)
\ar[d]
\\
\nhomog \ho (\wcal \Top)
\ar[r]^{R S^*}
&
\nhomog \ho (\jcal_0 \Top)
}
\]
\caption{The homotopy level diagram}
\label{fig:completehomotopylevel}
\end{figure}






\subsection{Restriction of excisive functors are polynomial}\label{sec:nexsnpoly}
In this section, we prove that the restriction of an $n$-excisive functor is $n$-polynomial. 

\begin{prop}\label{prop:nexc-north}
For $F$ a homotopy functor, the map $w_n: S^\ast (P_nF) \ra \wt_n (S^\ast (P_n F) )$ is an objectwise weak equivalence. Equivalently, restriction to $\jcal_0$ of an $n$-excisive functor yields an $n$-polynomial functor.
\end{prop}

\begin{proof}
We work by induction. If $n=0$ then $S^* P_0 F = \ast$, which is $0$-polynomial
in $\jcal_0 \Top$. 
Now we assume the result holds for $n-1$.
By \cite[Lemma 2.2]{goodcalc3}, there is a fibre sequence with 
$R_n F \in \wcal \Top$ an $n$-homogeneous in the sense of Goodwillie:
\[
P_n F \lra P_{n-1}F  \overset{\rho}{\lra} R_n F
\]
By Proposition \ref{prop:restrictionhomog}, we know  $S^* R_n F$ is $n$-homogeneous
in the orthogonal calculus.  
It follows that $\ind_0^{n+1} R_n F(S^V)$ is trivial.
By inductive hypothesis, 
$S^* P_{n-1} F$ is $(n-1)$-polynomial, and thus 
$n$-polynomial (by Proposition \ref{prop:n-is-n+1}). 
Hence, we may apply \cite[Lemma 5.5]{weiss95} to 
\[
S^* \rho \co S^* (P_{n-1} F) {\lra} S^* (R_n F)
\]
to conclude that its homotopy fibre is $n$-polynomial in the orthogonal calculus.
Since homotopy fibres are constructed objectwise, the homotopy fibre is 
$S^* (P_n F)$.
\end{proof}

\begin{remark}
Proposition \ref{prop:nexc-north} implies that if $F$ satisfies
$F(X) \overset{\simeq}{\lra} \holim_{k \in \mathscr{P}_0([n])} F(k \join X)$
for all $X \in \wcal$, then $S^\ast F$ satisfies
$F(S^V) \overset{\simeq}{\lra} \holim_{0 \neq U \subset \R^{n+1}} F(S^{V \oplus U})$ for all $V \in \jcal_0$.
This can be thought of as a kind of enriched cofinality.
\end{remark}


\subsection{Agreement of towers for restricted analytic functors}\label{sec:stablyexs}

It is observed in \cite{goodcalc3} that when $F$ is stably n-excisive, $P_nF$ is, up to natural equivalence, the only $n$-excisive functor that agrees to $n^{th}$ order with $F$. The uniqueness of $P_n F$ for these functors implies more than just $n$-polynomialness of $S^\ast (\P_n F)$ -- we can also say that $S^\ast (\P_n F)$ is the $n$-polynomial approximation of $S^\ast F$.  If we assume the stronger property of $\rho$-analyticity for $F$ (which implies that $F$ is stably $n$-excisive for all $n$), then we can conclude that the tower
$\{ \wt_n (S^* F)  \}_{n \geqslant 0}$
is equivalent to the tower
$\{ S^* (P_n F)  \}_{n \geqslant 0}$.





\begin{prop} \label{prop:samepieces}
If $F$ is a stably $n$-excisive  homotopy functor, 
the following map is a weak homotopy equivalence for all $V \in \jcal_0$.
\[
\wt_n(S^* p_{n})(V) \co \wt_n (S^\ast F)(V) \lra \wt_n (S^\ast P_nF)(V)
\]
Thus the $n$-polynomial approximation of $S*F$ is 
given by the map $S^*F \to S^*(P_n F)$. 
\end{prop}
\begin{proof}
For any $F \in \wcal \Top$ there is a commutative square as below. 
Proposition \ref{prop:nexc-north} implies that the right hand vertical map is a weak equivalence. 
\[
\xymatrix{
(S^\ast F)(V) \ar[rr]^{p_{n,S^V}}\ar[d] & &S^\ast(P_n F)(V) \ar[d]^\simeq\\
\wt_n (S^\ast F)(V) \ar[rr]^{\wt_n(p_{n,S^V})} && \wt_n S^\ast(P_nF)(V)
}
\]
Assume that $F \in \wcal \Top$ is stably $n$-excisive, 
then by \cite[Proposition 1.5]{goodcalc3}, the map $p_n: F \ra P_n F$
is an agreement to order $n$ (Definition \ref{def:ordern}). 
By Lemma \ref{lem:ordnorthordn}, we also know that the restriction to $\jcal_0$
\[
S^* p_{n} \co S^* F \lra S^*(P_nF)
\]
is an orth-order-$n$ agreement for all $V$ (Definition \ref{def:orthordn}).
Then, applying Proposition \ref{lem:ordnWneq}, we conclude that
\[
\wt_n( S^* p_{n}) \co \wt_n (S^* F) \lra \wt_n (S^*(P_n F)) 
\]
is an objectwise equivalence. The commutative square then shows that 
the $n$-polynomial approximation to $S^*F$, $S^* F \to \wt_n S^*F$, 
is weakly equivalent to $S^*F \to S^*(P_n F)$. 
\end{proof}

To extend to an equivalence of towers, not just of the $n^{th}$ level for a fixed $n$, we consider the stronger notion of $\rho$-analytic. 

\begin{theorem}\label{thm:tower-eq}
Consider $F$ a $\rho$-analytic homotopy functor for some $\rho$ and $S^\ast F$ its restriction to $\jcal_0$. Then  the Weiss tower of $S^\ast F$ is equivalent to the restriction to $\jcal_0$ of the Goodwillie tower of $F$.
\end{theorem}
\begin{proof}
Since $\rho$-analytic functors are stably $n$-excisive for all $n$ 
(Remark \ref{rem:analytic-stab}), we have a commutative diagram
as below. Moreover,  the horizontal maps are objectwise weak equivalences
by Propositions \ref{prop:nexc-north} and \ref{prop:samepieces}.
\[
\xymatrix@C+0.8cm{
S^\ast (P_nF) 
\ar[r]^-{w_n}_-{\simeq \text{ all }F} 
\ar[d]
& 
\wt_n (S^\ast (P_n F) )
\ar[d]
&&
\ar[ll]_-{\wt_n(S^* p_{n})}^-{\simeq { \text{ stably $n$-exc }F}}
\wt_n (S^\ast F) 
\ar[d]
\\
S^\ast (P_{n-1}F) 
\ar[r]^-{w_{n-1}}_-{\simeq \text{ all }F}  & 
\wt_{n-1} (S^\ast (P_{n-1} F) )
&&
\ar[ll]_-{\wt_n(S^* p_{{n-1}})}^-{\simeq {\text{ stably $(n-1)$-exc }F}}
\wt_{n-1} (S^\ast F)
}
\]
It follows that $S^*(D_n^G F)$ is objectwise weakly equivalent to 
$D_n^W S^* F$. 
Furthermore, by Proposition \ref{prop:restrictionhomog}, if $\Theta_F$ is the spectrum 
with $\Sigma_n$-action that corresponds to $D_n^G F$ (by Goodwillie's classification, 
see Theorem \ref{thm:classifications}), then 
$O(n)_+ \smashprod_{\Sigma_n}^L \Theta_F$ is the spectrum corresponding to 
$D_n^W S^* F$ (by Weiss's classification).
\end{proof}



While $S^\ast P_n F$ is $n$-polynomial, in general there is no reason for it to be the $n$-polynomial approximation to $S^*F$. 
Consequently, for non-analytic functors $F$, the two towers $\{S^\ast P_n F\}_{n \geq 0}$ and $\{\wt_n (S^\ast F)\}_{n \geq 0}$ do not need to be equivalent under $S^*$.

\begin{corollary}\label{cor:GSVanlytic}
If $F\in \wcal \Top$ is $\rho$-analytic, $S^\ast F$ agrees to orth-order-$n$ with $\wt_n (S^\ast F)$ for all $n$. Moreover, $S^\ast F$ is weakly $(\rho+1)$-analytic. That is, for $V$ at least dimension $(\rho+1)$, we have that $S^\ast F (V) \overset{\simeq}{\lra}\wt_\infty (S^\ast F)(V)=  \hocolim_n (\wt_n S^* F)(V)$.  
\end{corollary}

\begin{proof}
Since $F$ is $\rho$-analytic, we have weak equivalences for all $n$
\[
S^\ast(\P_nF) 
\overset{\simeq}{\lra}
\wt_n (S^\ast (\P_n F))  
\overset{\simeq}{\longleftarrow}
\wt_n (S^\ast F) 
\]
%
and we have the
commutative diagram of objects of $\jcal_0 \Top$ as in Figure \ref{fig:towerdiagram}. 
We let $q_n \co P_n \to P_{n-1}$ and $v_n \co \wt_n \to \wt_{n-1}$
be the natural transformations arising from Proposition \ref{prop:n-is-n+1}.
\begin{figure}[H]
\[
\xymatrix{
&& 
\ar[dll]_-{S^* p_n}
S^*F 
\ar[drr]^-{w_n}
&& \\
S^*P_n F 
\ar[rr]_-\simeq^{w_n}
\ar[dd]_-{S^* q_n}
& & 
\ar[dl]_-{\wt_n S^* q_n}
\wt_n S^*P_n F 
\ar[dr]_-{v_n}
&& 
\ar[ll]^-\simeq_-{\wt_n S^* p_n}
\wt_{n} S^* F
\ar[dd]_{v_n}\\
& 
\wt_n S^*P_{n-1} F 
\ar[dr]_-{v_n}
&&  
\wt_{n-1} S^*P_{n} F 
\ar[dl]_-{\wt_{n-1} S^* q_n}
& \\
S^*P_{n-1} F
\ar[ur]^{w_n}
\ar[rr]_{w_{n-1}}^\simeq
&& 
\wt_{n-1} S^*P_{n-1} F 
&& 
\ar[ul]^{\wt_{n-1}S^* p_{n}}
\ar[ll]^{\wt_{n-1}S^* p_{n-1}}_\simeq
\wt_{n-1} S^* F 
}
\]
\caption{Commutativity of the tower maps}
\label{fig:towerdiagram}
\end{figure}

The above figure implies that we have a commutative diagram of homotopy
limits as below, with the maps from $S^\ast F$ are weak equivalences only when $\dim(V) \geq \rho+1$ (so that $S^V$ is in $F$'s radius of convergence).  
\[
\xymatrix@C+0.5cm{
& \ar[dl]_-\simeq^-{S^*p} S^*F \ar[dr]^-\simeq_-w & \\
S^* \holim_n P_n F \ar[d]^= \ar[r]^-\simeq & 
\holim_n \wt_n S^* \holim_n P_n F \ar[d]^= & 
\holim_n \wt_n S^* F \ar[l]_-\simeq \ar[d]^= \\
S^* P_\infty F \ar[r]^\simeq& 
\wt_\infty S^* P_\infty F & 
\wt_\infty S^* F \ar[l]_\simeq
}
\]
In particular, $S^*F(V) \overset{\simeq}{\lra} \wt_\infty S^\ast F(V)$ for all $V$ of dimension larger than $\rho$, as desired. 
\end{proof}


\begin{ex}\label{ex:SVanalytic}
The identity functor $Id_{\Top} \in \wcal\Top$ is 1-analytic. Therefore, $V \mapsto S^V$ (which is equal to $S^\ast Id_{\Top}$) is weakly 2-analytic, by Corollary \ref{cor:GSVanlytic}. 
\end{ex}








\section{Application: Orthogonally weak analytic functors}\label{sec:orthogweaklytic}
In this section, we show that one can obtain weak analyticity results for functors in $\jcal_0 \Top$ whose unstable first derivative is known already to be analytic. Our immediate application is that the Weiss tower of $V \mapsto BO(V)$ converges to $V \mapsto BO(V)$ when $\dim(V) \geqslant 2$ and similarly for $V \mapsto BU(V)$ for $V$ of dimension $\geq 1$. 

\begin{theorem}\label{thm:weissanalytic}
Let $F$ be a functor in $\jcal_0 \Top$. Assume that its unstable $n$th orthogonal derivative,
the functor $V \mapsto \ind_0^1 F(V)$ in $\jcal_0 \Top$,
is weakly $\rho$-analytic. Assume, moreover, that for $V$ of dimension at least $\rho$, $F(V)$ is path-connected. Then $F$ is itself weakly $\rho$-analytic. 
\end{theorem}



\begin{corollary}\label{cor:BOV}
The functor $V \mapsto BO(V)$ is weakly $2$-analytic and 
the functor $V \mapsto BU(V)$ is weakly 1-analytic. 
\end{corollary}

This follows immediately from Theorem \ref{thm:weissanalytic}, using the facts $\ind_0^1 BO(V)\simeq S^V$ \cite[Example 2.7]{weiss95}, 
$\ind_0^1 BU(V) \simeq \Sigma S^V$ and $V \mapsto S^V$ is weakly $2$-analytic by Example \ref{ex:SVanalytic}.
Recall that for $\gcal$ a topological group 
$B\gcal=E\gcal/\gcal$ is path connected.

This particular analyticity bound matches well with observed behavior of BO -- that it does not behave well on $V$ with dimension less than 2 (see, for instance, Proposition 3.4 and 3.5 of Reis-Weiss \cite{weissreis}).

\begin{remark}
This is a formalization of the comment of Arone in \cite[p.13]{aroneweiss} that the Taylor tower of $BO$ converges. 
\end{remark}

\begin{proof}[of Thereom \ref{thm:weissanalytic}]
It suffices to argue with a skeleton of $\jcal_0$, that is, with $\R^n$ instead of arbitrary $V$ of dimension $n$. In the following, we let $\fcal (\rr^\infty)$ denote the homotopy colimit over $k$ of $\fcal (\rr^k)$ for $\fcal \in \jcal_0 \Top$. 

Recall that $F$ is weakly $\rho$-analytic if and only if  $F(V) \overset{\sim}{\lra} \wt_\infty F(V)$ 
for all $V$ of dimension at least $\rho$. 







For $F$ a functor in $\jcal_0 \Top$ there is a homotopy fibre sequence
\[
\xymatrix{
\ind_0^1 F(V) \ar[r]
&
F(V) \ar[r] 
& 
F(V \oplus \R) 
}
\]
by \cite[Proposition 2.2]{weiss95}. Since $\wt_\infty$ 
preserves fibre sequences, the natural transformation $F \ra \wt_\infty F$ 
yields a map of fibre sequences as in Figure \ref{fig:seqcart}. 

\begin{figure}[H]
\[
\xymatrix{
\ind_0^1 F (\R^\rho)\ar[d]\ar[r]& F(\R^\rho)\ar[d]\ar[r] & F(\R^{\rho+1})\ar[d]\\
\wt_\infty \ind_0^1 F (\R^\rho) \ar[r] & \wt_\infty F(\R^\rho) \ar[r] & \wt_\infty F(\R^{\rho+1}) \\ 
}
\]
\caption{map of fibre sequences}
\label{fig:seqcart}
\end{figure}

We now show that the right hand square is cartesian:
compare the long exact sequence of homotopy groups 
for (1) the homotopy fibre sequence that is the top of Figure 3
and (2) the homotopy fibre sequence that is
\[
\ind_0^1F(\R^\rho) 
\lra
Q_\rho
\lra
F(\R^{\rho+1}).
\]
Here, $Q_\rho$ is the homotopy pullback of
\[
\wt_\infty F(\R^\rho) \lra
\wt_\infty F(\R^{\rho+1})
\longleftarrow
F(\R^{\rho+1}).
\]
The five lemma applied to the long exact sequences of homotopy
groups associated to (1) and (2)
and a simple diagram chase for $\pi_0$ (using 
path-connectedness of $F(\rr^\rho)$ and  $F(\rr^{\rho+1})$) 
shows that $F(\rr^\rho) \to Q_\rho$ is a weak homotopy equivalence.
Consequently, the right hand square of Figure \ref{fig:seqcart} is cartesian as desired.


We can extend the right hand square of Figure \ref{fig:seqcart} to the right to get a shifted copy of the square. 
By the same argument, this is also cartesian. We can repeat this process to obtain 


\[
\xymatrix{
F(\R^\rho)\ar[d]\ar[r] \ar@{}[dr]|{\text{hoPB}}& F(\R^{\rho+1})\ar[d]\ar[r]\ar@{}[dr]|{\text{hoPB}}& F(\R^{\rho+2})\ar[d] \ar[r] & \cdots\\
\wt_\infty F(\R^\rho) \ar[r] & \wt_\infty F(\R^{\rho+1}) \ar[r] & \wt_\infty F(\R^{\rho+2}) \ar[r] & \cdots\\ 
}.
\]

Since the juxtaposition of two cartesian squares is cartesian, for all $n \geqslant \rho$ and $k \geqslant 1$ we have a cartesian square as on the left of the diagram below. 
Filtered homotopy colimits preserve cartesian
squares, hence the square on the right of the diagram below is cartesian. 
\[
\xymatrix{
F(\R^n) \ar[d] \ar[r] & F(\R^{n+k}) \ar[d] 
&&
F(\R^n)\ar[r]\ar[d] & F(\R^\infty) \ar[d]\\
\wt_{\infty} F(\R^{n})  \ar[r]     & \wt_{\infty} F(\R^{n+k})
&&
\wt_\infty F(\R^n) \ar[r] & \wt_\infty F(\R^\infty)
}
\]
For any $\fcal \in \mathcal{J}_0\Top$ and any $q \geq 0$, the map $\fcal (\R^\infty) \lra \wt_q \fcal (\R^\infty)$ is a weak homotopy equivalence by \cite[Lemma 5.14]{weiss95}.
That is, the (orthogonal) Taylor tower of $\fcal$ at $\R^\infty$ is constant.   
Therefore, $\fcal(\R^\infty) \overset{\simeq}{\lra}\wt_\infty \fcal(\R^\infty)$ and in particular, this holds for $\fcal = F$.
Thus the right hand vertical in the right hand square above is a weak homotopy equivalence. Since the square is cartesian, the left hand vertical is a weak homotopy equivalence for any $n \geqslant \rho$ as desired.
\end{proof}

\section{Application: Model Category Comparison}\label{sec:diagram}

In the previous sections we have examined the relations
between the orthogonal calculus and the homotopy
functor calculus at the level of homotopy categories.
We now lift this to the level of model structures
and construct the diagram of model categories
and Quillen functors as in Figure \ref{fig:big}.
We will then show that it commutes, in the sense that certain
compositions of functors agree up to natural isomorphism.
The notation here has been chosen to match the preceding paper \cite{barneseldred}.
The next few sections introduce the categories and model structures used
and then we turn to showing that the squares of the diagram commute.

\begin{figure}[H]
\[
\xymatrix@C+0.8cm@R+1cm{
\Sigma_n\circlearrowleft \wcal \Sp
\ar@<-5pt>[r]_{S^*}
\ar@<+5pt>[d]^-{\mu_n^\ast}
\ar @{} [dr] |{1}
&
\ar@<-5pt>[l]_{L_S}
\Sigma_n\circlearrowleft \jcal_1 \Sp
\ar@<+5pt>[r]^{O(n)_+ \smashprod_{\Sigma_n} (-)}
\ar@<+5pt>[d]^-{\alpha_n^\ast}
\ar @{} [dr] |{2}
&
\ar@<+5pt>[l]^-{i^*}
O(n) \circlearrowleft \jcal_1 \Sp
\ar@<+5pt>[d]^-{\alpha_n^\ast}
\\
\ar@<+5pt>[u]^-{\wcal \smashprod_{\wcal_n} -}
\devcat
\ar@<-5pt>[r]_{nS^*}
\ar@<-5pt>[d]_-{(-)/\Sigma_n \circ \mapdiag^*}
\ar @{} [drr] |{3}
&
\ar@<-5pt>[l]_{L_{nS}}
\ar@<+5pt>[u]^-{i^* \jcal_1 \smashprod_{i^* \jcal_n} -}
\interorthdevcat
\ar@<+5pt>[r]^{O(n)_+ \smashprod_{\Sigma_n} (-)}
&
\ar@<+5pt>[u]^-{\jcal_1 \smashprod_{\jcal_n} -}
\ar@<+5pt>[l]^-{i^*}
\orthdevcat
\ar@<-5pt>[d]_-{(-)/O(n) \circ \res_0^n}
\\
\ar@<-5pt>[u]_-{\diff_n}
\wcal \Top_{\nhomog}
\ar@<-5pt>[d]_{\id}
\ar@<-5pt>[rr]_{S^*}
\ar @{} [drr] |{4}
&&
\ar@<-5pt>[ll]_{L_S}
\ar@<-5pt>[u]_{\ind_0^n}
\jcal_0 \Top_{\nhomog}
\ar@<-5pt>[d]_{\id}
\\
\ar@<-5pt>[u]_-{\id}
\wcal \Top_{\nexs}
\ar@<-5pt>[d]_-{\id}
\ar@<-5pt>[rr]_{S^*}
\ar @{} [drr] |{5}
&&
\ar@<-5pt>[ll]_{L_S}
\ar@<-5pt>[u]_{\id}
\jcal_0 \Top_{\npoly}
\ar@<-5pt>[d]_{\id}
\\
\ar@<-5pt>[u]_-{\id}
\wcal \Top_{(n-1)\textrm{-exs}}
\ar@<-5pt>[rr]_{S^*}
&&
\ar@<-5pt>[u]_{\id}
\ar@<-5pt>[ll]_{L_S}
\jcal_0 \Top_{(n-1)\textrm{-poly}}
}
\]
\caption{The diagram of model categories}
\label{fig:big}
\end{figure}


%
%
%
%
%
%
%
%
%
%
%

\subsection{The functor \texorpdfstring{$S^*$}{S*}}
Here we recap the definition of the model categories needed for the
two forms of calculus and show that $S^*$ can be
given the structure of a Quillen functor. Thus squares 4 and 5 of
Figure \ref{fig:big} are commutative squares of Quillen functors.

Recall the functor  $S$ from $\jcal_0$ to $\wcal$, which sends
$V$ to the one-point compactification $S^V$. Given $F \in \wcal \Top$
we can pre-compose with $S$ to obtain $S^*F:=F \circ S \co \jcal_0 \to \Top$.
This functor has a left adjoint,
called $L_S$, which is given by the formula below.
\[
(L_S E)(A) = \int^{V \in \jcal_0} E(V) \smashprod \wcal(S^V,A)
\]
In \cite{mmss01}, $S$ is called $\mathbb{U}$ and its left adjoint is denoted $\mathbb{P}$.

\begin{proposition}\label{prop:levelwisemodel}
  There is a \textbf{objectwise model structures} on $\wcal \Top$ and $\jcal_0 \Top$
  whose fibrations and weak equivalences are defined objectwise. It is proper and cofibrantly generated. 
  \end{proposition}

For the cross effect (Definition \ref{def:cref}) to be a right Quillen functor requires more cofibrations than in the objectwise model structure on $\wcal \Top$.
We now specify the extra maps which are needed: 

\begin{definition}\label{def:maps}
Consider the following collection of maps, where $\phi_{\underline{X},n}$ is defined via the projections which send those factors in $S$ to the basepoint.
\[
\begin{array}{c}
\Phi_n = \{  \phi_{\underline{X},n}
\co
\underset{S \in \pcal_0(\underline{n})}{\colim}
\wcal({\bigvee_{l \in \underline{n} - S} X_l}, -)
\longrightarrow
\wcal({\bigvee_{l =1}^n X_l},-)
\ | \ \underline{X} = (X_1, \dots , X_n), X_l \in \skel \wcal  \} \\
\end{array}
\]
We then also define $\Phi_\infty = \cup_{n \geqslant 1} \Phi_n$.
\end{definition}
The cofibre of $\nat(-,F) (\phi_{\underline{X},n})$ is the cross effect
of $F$ at $\underline{X}$, $\cref_n(F)(X_1, \ldots, X_n)$;
see \cite[Lemma 3.14]{BRgoodwillie}.

\begin{definition}
Given $f \co A \to B$ a map of based of spaces and $g \co X \to Y$ in $\wcal$,
the \textbf{pushout product} of $f$ and $g$, $f \square g$, is given by
\[
f \square g \co
B \smsh X \bigvee_{A \smsh X} A \smsh Y \to B \smsh Y.
\]
\end{definition}

  \begin{proposition}\label{prop:crosseffectmodel}
  There is a proper, cofibrantly generated model structure on $\wcal \Top$, \textbf{the cross effect model structure},
whose weak equivalences are the objectwise weak homotopy equivalences
and whose generating sets are given by
\[
I_{\wcal cr}  = \Phi_\infty \square I_{\Top} \quad J_{\wcal cr} = \Phi_\infty \square J_{\Top}
\]
  Every cross effect fibration is in particular a objectwise fibration.
\end{proposition}
\begin{proof}
See \cite[Theorem 3.6]{barneseldred} and \cite[Lemma 6.1]{barnesoman13}.
\end{proof}
The point of the cross effect model structure is that it allows
$\diff_n$ to be a right Quillen functor, see \cite[Proposition 6.3]{barneseldred}.
Now that we have our initial model structures we can show that $S^*$ is a right Quillen functor. 

\begin{lemma}
  The functor $S^*$ is a right Quillen functor when $\wcal \Top$ and $\jcal_0 \Top$
  are both equipped with the objectwise model structures. Furthermore it is a
  right Quillen functor when $\wcal \Top$ has the cross effect model structure and
  $\jcal_0 \Top$ is equipped with the objectwise model structure.
\end{lemma}
\begin{proof}
  Take some (acyclic) fibration $f$ in the objectwise model structure on $\wcal \Top$.
  Then $f(A)$ is a (acyclic) fibration of spaces for any $A \in \wcal$.
  Hence $f(S^V)$ is a (acyclic) fibration of spaces for any $V \in \jcal_0$.
  This proves the first statement. For the second, we note that the
  identity functor is a right Quillen functor from the cross effect model structure 
  to the objectwise model structure, since every cross effect fibration is, in particular, an objectwise fibration and the model structures have the same weak equivalences.
\end{proof}

In Square 5 of Figure \ref{fig:big} we are interested in the $n$-excisive and 
$n$-polynomial model structures, which we introduce below. 
For the proof, see \cite[Theorem 3.14]{barneseldred} and \cite[Proposition 6.5]{barnesoman13}.

\begin{proposition}
  There is an \textbf{$n$-excisive model structure} on $\wcal \Top$ 
  which has the same cofibrations as the cross effect model structure and 
  whose weak equivalences are those maps $f$ such that 
  $P_n f$ is an objectwise weak equivalence. 
  The fibrant objects are the $n$-excisive functors that are fibrant
  in the cross effect model structure.  
  A map $f: X \ra Y$ is an $n$-excisive fibration if and only if it is 
  a cross effect fibration and the square below is cartesian.
  \[
  \xymatrix{
  X \ar[r] \ar[d] & P_n X \ar[d]\\
  Y \ar[r] & P_n Y 
  } 
 \]

  There is an \textbf{$n$-polynomial  model structure} on $\jcal_0 \Top$ 
  which has the same cofibrations as the objectwise model structure and
  whose weak equivalences are those maps $f$ such that
  $\wt_n f$ is a objectwise weak equivalence.
  The fibrant objects are the $n$-polynomial functors.

  Both of these model structures are proper and cofibrantly generated.
\end{proposition}

The functor $S^*$ remains a right Quillen functor with respect to these
model structures.

\begin{proposition}
  The functor $S^*$ is a right Quillen functor between $\wcal \Top$
  with the $n$-excisive model structure and $\jcal_0 \Top$ with
  the $n$-polynomial model structure.
\end{proposition}
\begin{proof}
We must show that the Quillen pair of the previous lemma
induces a Quillen pair between the model categories
$\wcal \Top_{\nexs}$ and $\jcal_0 \Top_{\npoly}$.
We know that the functor $S^*$ is a right Quillen functor from
$\wcal \Top$ to
$\jcal_0 \Top$, where both categories have their objectwise model structures.
Similarly the identity functor from $\wcal \Top_{\nexs}$ to
$\wcal \Top$ (with the objectwise model structure) is a right Quillen functor.
The result then follows by \cite[Theorem 3.1.6]{hir03} and
Section \ref{sec:nexsnpoly}, which proves that $S^*$ sends
$n$-excisive functors to $n$-polynomial functors.
\end{proof}

\begin{corollary}
The derived functor of $S^*$ from $\wcal \Top_{\nexs}$
to $\jcal_0 \Top_{\npoly}$ is given by $S^* \circ P_n$.
When restricted to stably $n$-excisive functors, 
the derived functor is given by $S^*$ (by Proposition \ref{prop:samepieces}). 
\end{corollary} 

\begin{remark}\label{rmk:directproof}
Proposition \ref{prop:nexc-north} implies that 
the left adjoint $\wcal \smashprod_{\jcal_0} (-)$
takes maps of the form
\[
\delta_{n,U,V} \co
\hocolim_{0 \neq U \subseteq \rr^{n+1}}
\jcal_0( U  \oplus V, - )
\longrightarrow
\jcal_0( V, - )
\]
(for $V \in \jcal_0$) to $P_n$-equivalences in $\wcal \Top$.
That is, 
\[
\wcal \smashprod_{\jcal_0} \delta_{n,U,V} \co
\hocolim_{0 \neq U \subseteq \rr^{n+1}}
\wcal( S^U \smashprod S^V, - )
\longrightarrow
\wcal( S^V, - )
\]
is a $P_n$-equivalence for 
all $V \in \jcal_0 \Top$.
\end{remark}

In Square 4 of Figure \ref{fig:big} we have $n$-homogeneous model structures, which we introduce below. For the details, see \cite[Theorem 3.18]{barneseldred} and \cite[Proposition 6.9]{barnesoman13}.

\begin{proposition}
  There is an \textbf{$n$-homogeneous model structure} on $\wcal \Top$
  which has the same fibrations as the $n$-excisive model structure and
  whose weak equivalences are those maps $f$ such that
  $D_n^G f$ is an objectwise weak equivalence.
  The cofibrant-fibrant objects are the $n$-homogeneous functors that are 
  cofibrant-fibrant in the cross effect model structure.

  There is an \textbf{$n$-homogeneous model structure} on $\jcal_0 \Top$
  which has the same fibrations as the $n$-excisive model structure and
  whose weak equivalences are those maps $f$ such that
  $D_n^W f$ is an objectwise weak equivalence.
  The cofibrant-fibrant objects are the $n$-homogeneous functors that are
  cofibrant in the objectwise model structure.

  Both of these model structures are proper, stable and cofibrantly generated.
\end{proposition}


The functor $S*$ is not just a right Quillen functor for $n$-polynomial functors, but also for $n$-homogeneous.

\begin{theorem}
  The functor $S^*$ is a right Quillen functor from $\wcal \Top$
  (equipped with the $n$-homogeneous model structure) to $\jcal_0 \Top$
  (equipped with the $n$-homogeneous model structure).
\end{theorem}
\begin{proof}
Since the $n$-homogeneous fibrations are precisely the
$n$-excisive (or $n$-polynomial) fibrations, $S^*$ preserves
fibrations.
Let $f \co X \to Y$ in $\wcal \Top$
be an acyclic fibration in the $n$-homogeneous model
structure.
By \cite[Lemma 6.22]{BRgoodwillie} $f$ is
an $(n-1)$-excisive fibration.
Thus $S^* f$ is an $(n-1)$-polynomial fibration
and the homotopy fibre of $S^* f$ is $(n-1)$-polynomial.
The model category $\jcal_0 \Top_{\nhomog}$ is stable and
$(n-1)$-polynomial functors are trivial in this model structure,
hence $S^*f$ is a weak equivalence in $\jcal_0 \Top_{\nhomog}$.
\end{proof}

\subsection{Comparisons between categories of spectra}\label{sec:compcatspectra}
Recall from Definition \ref{def:jcaln} that $\jcal_n$ is an enriched category of finite dimensional real inner product spaces with the space of morphisms the Thom space of the following vector bundle: 
\[
\gamma_n (U,V) = \{ (f,x) | f: U \to V,   x \in \rr^n \otimes (V-f(U))\}
\]
In this section, we compare the three categories of spectra with group action: 
$\Sigma_n\circlearrowleft \wcal \Sp$, $\Sigma_n\circlearrowleft \jcal_1 \Sp$ and
$O(n) \circlearrowleft \jcal_1 \Sp$. 
In terms of the diagram in Figure \ref{fig:big}, this comparison is the top row:
\[
\xymatrix@C+0.8cm@R+1cm{
\Sigma_n\circlearrowleft \wcal \Sp
\ar@<-5pt>[r]_{S^*}
&
\ar@<-5pt>[l]_{L_S}
\Sigma_n\circlearrowleft \jcal_1 \Sp
\ar@<+5pt>[r]^{O(n)_+ \smashprod_{\Sigma_n} (-)}
&
\ar@<+5pt>[l]^-{i^*}
O(n) \circlearrowleft \jcal_1 \Sp
\\
}
\]
In each case the notation $\Sp$ is to remind us that the stable model structure is being used.

Recall that an orthogonal spectrum is simply a continuous functor from $\jcal_1$ to $\Top$
and a $\wcal$-spectrum is a continuous functor from $\wcal$ to $\Top$. In both cases, we can forget structure to obtain a \textbf{sequential spectrum}:
a collection of spaces $\{ X_n \}_{n \geqslant 0}$ with maps
$X_n \smashprod S^1 \to X_{n+1}$. In the orthogonal case, we
set $X_n$ to be the functor evaluated at $\rr^n$, in the $\wcal$ case we set
it to be the functor evaluated at $S^n$.

\begin{proposition}
  For $\gcal=\Sigma_n$ or $O(n)$ and $\dcal=\jcal_1$ or $\wcal$, $\gcal \circlearrowleft \dcal \Sp$ denotes the category of $\gcal$-objects and $\gcal$-equivariant maps in $\dcal\Top$, with the \textbf{stable model structure}. The weak equivalences are those maps
  which forget to $\pi_*$-isomorphisms of non-equivariant fibre spectra. 

  The generating cofibrations are given by 
  $\gcal_+ \smashprod \dcal(d,-) \smashprod i$
  where $d$ is an element of a skeleton of $\dcal$ and $i$ is a generating
  cofibration of based topological spaces. 
  The cofibrant objects are $\gcal$-free (they have no $\gcal$-fixed points). 
  These model structures are proper, stable and cofibrantly generated.
\end{proposition}
\begin{proof}
  The non-equivariant stable model structures $\wcal \Sp$ and 
  $\jcal_1 \Sp$  exist by \cite[Theorem 9.2]{mmss01}. The equivariant
  versions exist by applying the transfer argument of Hirschorn 
  \cite[Theorem 11.3.2]{hir03} to the free functor 
  $\gcal_+ \smashprod -$ from $\dcal \Sp$ to $\gcal \circlearrowleft \dcal \Sp$.
\end{proof}

The one-point compactification construction induces a functor of enriched categories
$S \co \jcal_1 \to \wcal$. The mapping spaces of $\jcal_1$ admit the
description  (see \cite[Definition 4.1]{mm02} for more details)
\[
\jcal_1(U,V) = O(V)_+ \smashprod_{O(V-U)} S^{V-U}
\]
Note that this requires us to choose some preferred inclusion $U \hookrightarrow V$.
We want to define a map of spaces $\jcal_1(U,V) \to \wcal(S^U,S^V)$, it suffices (due to adjointness)
to construct a map
\[
O(V)_+ \smashprod_{O(V-U)} S^{V-U} \smashprod S^U \to S^V.
\]
To do so, we simply compose the isomorphism $S^{V-U} \smashprod S^U \to S^V$
with the action map of $O(V)$ on $S^V$.

Hence precomposition with $S$ is a functor $S^* \co \wcal \Top \to \jcal_1 \Top$. This induces a functor between the respective categories of $\Sigma_n$-objects and $\Sigma_n$-equivariant maps.

\begin{lemma}\label{lem:squillen}
  The adjoint pair $(L_S, S^*)$ is a Quillen equivalence between
  $\Sigma_n\circlearrowleft \wcal \Sp$ and $\Sigma_n\circlearrowleft \jcal_1 \Sp$.
\end{lemma}
\begin{proof}
  The non-equivariant statement is \cite[Theorem 0.1]{mmss01}.
  The equivariant result follows from the fact that the fibrations (weak equivalences)
  of both model categories are defined by forgetting to the non-equivariant versions.
\end{proof}

We now compare $\Sigma_n\circlearrowleft \jcal_1 \Sp$ and $O(n) \circlearrowleft \jcal_1 \Sp$.
Given an object in $O(n) \circlearrowleft \jcal_1 \Sp$ we may forget some of the action
and obtain a $\Sigma_n$-object. We call this functor $i^*$. It has a left 
adjoint given by $O(n)_+ \smashprod_{\Sigma_n} -$. 
See \cite[Section 2]{mm02} for more details (in the terms of that reference, we are using the trivial universe $\rr^\infty$ for our equivariant spectra).


\begin{lemma}\label{lem:istarquillen}
  The functor $i^*$ is a right  Quillen functor (with respect to the
  stable model structures) from
  $O(n) \circlearrowleft \jcal_1 \Sp$ to $\Sigma_n\circlearrowleft \jcal_1 \Sp$.
\end{lemma}
\begin{proof}
  That $i^*$ preserves fibrations and weak equivalences is immediate: both classes
  are defined by forgetting the group actions entirely.
\end{proof}

We have now shown that the top row of Figure \ref{fig:big} consists of
Quillen functors between model categories.

\subsection{Comparisons between the intermediate categories}
We now move on to the second row of Figure \ref{fig:big},
which compares the various intermediate categories.
\[
\xymatrix@C+0.8cm@R+1cm{
\devcat
\ar@<-5pt>[r]_{nS^*}
&
\ar@<-5pt>[l]_{L_{nS}}
\interorthdevcat
\ar@<+5pt>[r]^{O(n)_+ \smashprod_{\Sigma_n} (-)}
&
\ar@<+5pt>[l]^-{i^*}
\orthdevcat
\\
}
\]
This section is similar to the previous one, but with more complicated 
enriched categories $\jcal_n$ and $\wcal_n$, see 
Definition \ref{def:jcaln} and Definition \ref{def:wcaln}. 
\begin{definition}\label{def:devorthcat}
We define $\devcat$ as the category of $\Sigma_n \lca \Top$-enriched functors from
$\wcal_n$ to $\Sigma_n \lca \Top$.
Similarly, $\orthdevcat$ is the category of $O(n) \lca \Top$-enriched functors from
$\jcal_n$ to $O(n) \lca \Top$.
By forgetting structure, we obtain a $\Sigma \lca \Top$-enriched category $i^* \jcal_n$.
We define $\interorthdevcat$ to be the category of $\Sigma_n \lca \Top$-enriched functors from $i^* \jcal_n$ to $\Sigma \lca \Top$.
\end{definition}

We first compare $\devcat$ and $\interorthdevcat$, via an adjunction induced by
 a map of enriched categories $nS \co i^\ast \jcal_n \to \wcal_n$, which we construct below.  

Recall (Definition \ref{def:jcaln}) that $\jcal_n(U,V)$ is defined as the Thom space of the vector bundle 
\[
\gamma_n (U,V) = \{ (f,x) \ | \ f: U \to V, \  x \in \rr^n \otimes (V-f(U))\}
\]
over the space of linear isometries from $U$ to $V$, $\lcal(U,V)$.
Note that $\jcal_0(U,V)= \lcal(U,V)_+$.


Projection onto factor $l$, i.e. the map from $\rr^n \to \rr$ via $(x_1, \dots x_n) \mapsto x_l$),
induces a morphism of vector bundles
$\gamma_n(U,V) \to \gamma_1(U,V)$. 
This map is $\Sigma_n$-equivariant.
\[
\xymatrix{
\gamma_n(U,V) \ar[r] \ar[d] &
\gamma_1(U,V) \times \dots \times \gamma_1(U,V) \ar[d] \\
\lcal(U,V) \ar[r] &
\lcal(U,V) \times \dots \times \lcal(U,V)
}
\]
Taking the induced map on Thom spaces gives a $\Sigma_n$-equivariant
morphism as below, see \cite[Chapter 15, Proposition 1.5]{huse94}.
\[
i^* \jcal_n(U,V) \longrightarrow
\bigsmashprod{k=1}{n}  \jcal_1(U,V)
\]
Composing this map with the $n$-fold smash of the map $S$ gives a $\Sigma_n$-equivariant map
\[
i^* \jcal_n(U,V) \longrightarrow
\bigsmashprod{k=1}{n}  \wcal(S^U,S^V).
\]
Hence we have a map of $\Sigma_n \lca \Top$-categories $i^* \jcal_n \to \wcal_n$, which we call $nS$. 


\begin{definition}
Let $F$ be an object of $\devcat$, precomposing $F$ with $nS$ gives a functor
$ns^*F:=F \circ nS$ from $i^* \jcal_n$ to $\Sigma_n \lca \Top$.  Hence,
$nS^*$ is a functor from $\devcat$ to $\interorthdevcat$. This has a left adjoint given by the formula
\[
(L_{nS} E)(A) = \int^{V \in \jcal_0} E(V) \smashprod \wcal_n(S^V,A).
\]
\end{definition}

As with spectra, we have change of groups adjunctions, which are closely related
to the constructions of \cite[Section 2]{mm02}. Changing from 
$\Sigma_n$-equivariance to $O(n)$-equivariance gives us an adjunction between $\interorthdevcat$ and $\orthdevcat$.

\begin{definition}
Let $E \in \orthdevcat$, then 
$E$ has $O(n)$-equivariant structure maps 
\[
E_{U,V} \co \jcal_n(U,V) \to \Top(E(U),E(V)).
\]
We can forget structure to obtain a $\Sigma_n$-equivariant map
\[
i^* E_{U,V} \co i^* \jcal_,n(U,V) \to i^* \Top(E(U),E(V)) = \Top(i^* E(U),i^* E(V)).
\]
This gives a functor $i^* \co \orthdevcat \to \interorthdevcat$.

This functor has a left adjoint, given by applying $O(n) \smashprod_{\Sigma_n} -$ objectwise to
an object of $\interorthdevcat$: 
$(O(n) \smashprod_{\Sigma_n} F) (U) = O(n) \smashprod_{\Sigma_n} F(U) $.
The structure map is given below, where the first isomorphism is
a standard result about group actions and the second is $O(n)_+ \smashprod_{\Sigma_n}- $
applied to the structure map of $F$.
\[
O(n) \smashprod_{\Sigma_n} F (U) \smashprod \jcal_n (U,V)
\cong
O(n) \smashprod_{\Sigma_n} (F (U) \smashprod i^* \jcal_n (U,V) )
\to
O(n) \smashprod_{\Sigma_n} F (V)
\]
\end{definition}

Just as with $\devcat$ and $\orthdevcat$, we can put an objectwise and an $n$-stable
model structure on $\interorthdevcat$.

\begin{lemma}
  The categories $\interorthdevcat$, $\orthdevcat$ and $\devcat$ admit 
  \textbf{objectwise model structures}
  where the weak equivalences and fibrations are those maps $f$ such that
  $f(A)$ is a weak equivalence or fibrations of based (non-equivariant) spaces
  for each $A$ in $i^* \jcal_n$, $\jcal_n$ or $\wcal_n$. 
  These model structures are cofibrantly generated and proper. 
\end{lemma}

\begin{proposition}
  Each of the categories $\interorthdevcat$, $\orthdevcat$ and $\devcat$ 
  admits an \textbf{$n$-stable model structure}
  which is the left Bousfield localisation of the objectwise model structure at the set of maps
  \[
  \begin{array}{rcl}
  i^* \jcal_n(U \oplus \rr,-) \smashprod S^n
  \longrightarrow
  i^* \jcal_n(U ,-) & \quad \textrm{for } & \interorthdevcat \\
  \jcal_n(U \oplus \rr,-) \smashprod S^n
  \longrightarrow
  \jcal_n(U ,-) & \quad \textrm{for } & \orthdevcat \\
  \wcal_n(X \smashprod S^1,-) \smashprod S^n
  \longrightarrow
  \wcal_n(X ,-) & \quad \textrm{for } & \devcat \\
  \end{array}
  \]
  as $U$ (or $X$) runs over the objects of a skeleton for $i^* \jcal_n$ (or $\wcal_n$).
  The fibrant objects of $\interorthdevcat$ are those $E$ such that
  $E(V) \to \Omega^n E(V \oplus \rr)$ is a weak equivalence of spaces
  for all $V \in i^* \jcal_n$. Similar statements hold for 
  $\orthdevcat$ and $\devcat$. 
  These model structures are proper, stable and cofibrantly generated.
\end{proposition}
\begin{proof}
  See \cite[Proposition 4.12]{barneseldred} and \cite[Proposition 7.14]{barnesoman13}
  for $\devcat$ and $\orthdevcat$. The case of $\interorthdevcat$ is similar. 
\end{proof}

The weak equivalences of these categories are called
$n \pi_*$-isomorphisms and are similar to $\pi_*$-isomorphisms of spectra.
For example, let $F  \in \devcat$, then the $n \pi_p$-group of $F$
is defined in terms of the colimit of the system below.
\[
\pi_{p+nk} (F(S^k) )
\lra 
\pi_{p+n{k+1}} (F(S^{k}) \smashprod S^n)
\lra
\pi_{p+n{k+1}} (F(S^{k+1}) )
\]
For more details, see \cite[Section 4.3]{barneseldred} and 
\cite[Section 7]{barnesoman13}.


\begin{lemma}\label{lem:nsquillen}
  The adjoint pair $(L_{nS},nS^*)$ are a Quillen pair between the stable model structures
  on $\devcat$ and $\interorthdevcat$.
\end{lemma}
\begin{proof}
  It is clear that $nS^*$ preserves objectwise fibrations and weak equivalences.
  In \cite{barneseldred} the fibrant objects of the $n$-stable model structure on
  $\devcat$ are identified as those $F$ such that
  $F(A) \to \Omega^n F(A \smashprod S^1)$ is a weak equivalence for all $A \in \wcal$.
  It is clear that for such an $F$, $S^* F$ is fibrant in $\interorthdevcat$.
  It follows that we have a Quillen pair between the $n$-stable model structures.
\end{proof}

\begin{lemma}
  The functor $i^*$ is a right Quillen functor.
\end{lemma}
\begin{proof}
  The functor $i^*$ preserves objectwise fibrations and weak equivalences.
  It is easy to check that it also preserves the weak equivalences
  of the $n$-stable model structure. 
\end{proof}

\subsection{Squares 1 and 2 commute}\label{subsec:square1and2}

We show that Squares 1 and 2 of Figure \ref{fig:big} (shown again below) commute.
We have already defined the model categories and horizontal adjunctions,
the vertical adjunctions are defined later in this section.
\[
\xymatrix@C+0.8cm@R+1cm{
\Sigma_n\circlearrowleft \wcal \Sp
\ar@<-5pt>[r]_{S^*}
\ar@<+5pt>[d]^-{\mu_n^\ast}
\ar @{} [dr] |{1}
&
\ar@<-5pt>[l]_{L_S}
\Sigma_n\circlearrowleft \jcal_1 \Sp
\ar@<+5pt>[d]^-{\alpha_n^\ast}
\ar@<+5pt>[r]^{O(n)_+ \smashprod_{\Sigma_n} (-)}
\ar @{} [dr] |{2}
& 
\ar@<+5pt>[l]^-{i^*}
O(n) \circlearrowleft \jcal_1 \Sp
\ar@<+5pt>[d]^-{\alpha_n^\ast}
\\
\ar@<+5pt>[u]^-{\wcal \smashprod_{\wcal_n} -}
\devcat
\ar@<-5pt>[r]_{nS^*}
&
\ar@<-5pt>[l]_{L_{nS}}
\ar@<+5pt>[u]^-{i^* \jcal_1 \smashprod_{i^* \jcal_n} -}
\interorthdevcat
\ar@<+5pt>[r]^{O(n)_+ \smashprod_{\Sigma_n} (-)}
&
\ar@<+5pt>[u]^-{\jcal_1 \smashprod_{\jcal_n} -}
\ar@<+5pt>[l]^-{i^*}
\orthdevcat
}
\]
To prove the commutativity , we will make use of a commutative square of 
$\Sigma_n \lca \Top$-enriched categories. 
We will give the diagram first, then define the morphisms
used.
\[
\xymatrix@C+1cm{
\jcal_0 \ar[r]^{i_n} \ar[d]^{S}
\ar@/^20pt/[rr]^{\R^n \ox-} &
\jcal_n
\ar[r]^{\alpha_n} \ar[d]^{nS} &
\jcal_1
\ar[d]^{S}\\
\wcal
\ar[r]^{\mapdiag}\ar@/_20pt/[rr]_{(-)^n} &
\wcal_n
\ar[r]^{\mu_n} &
\wcal \\
}
\]

On objects the map $i_n \co \jcal_0 \to \jcal_n$ is the identity, on morphism spaces it is induced by the morphism $\ical(U,V) \to \gamma_n(U,V)$
which sends $f$ to the pair $(f,0)$.

The map $\alpha_n$ sends a vector space
$U$ to $nU= \rr^n \otimes U$. On morphism spaces $\alpha_n$ is induced by the map which takes $(f,x) \in \gamma_n(U,V)$ to
$(n f, x) \in \gamma_1(nU,nV)$.

The morphism $\mapdiag$ is the identity on objects and sends
$\wcal(A,B)$ to $\wcal_n(A,B)$ by the diagonal. The morphism $\mu_n$
sends $A$ to $A^{\smashprod n}$ and on morphism spaces acts as the smash product.

On objects, the diagram clearly commutes up to natural isomorphism. We must now show that the following diagram of morphism spaces commutes.
\[
\xymatrix@C+1cm{
\jcal_0(U,V) \ar[r]^{i_n} \ar[d]^{S}
\ar@/^20pt/[rr]^{\R^n \ox-} &
\jcal_n(U,V)
\ar[r]^{\alpha_n} \ar[d]^{nS} &
\jcal_1 (nU,nV)
\ar[d]^{S}\\
\wcal (S^U,S^V)
\ar[r]^{\mapdiag}\ar@/_20pt/[rr]_{(-)^n} &
\wcal_n (S^U,S^V)
\ar[r]^{\mu_n} &
\wcal (S^{nU},S^{nV})\\
}
\]



Starting with some map $f \in \jcal_0(U,V)$,
it is easily checked that
\[
nS(i_n f) = (f, \dots, f) = \mapdiag (Sf)
\]
and hence the first square commutes.
The second square is more complicated, but it 
suffices to consider a point of 
$(f,x) \in \gamma_n(U,V)$. 
We have the following equality of maps
from $S^{\rr^n \otimes U}$ to $S^{\rr^n \otimes V}$,
where $f \co U \to V$ is an isometry and $x \in \rr^n \otimes (V - f(U)$.
\[
S(\alpha_n (f,x)) = S((nf,x)) 
\]
The above terms are equal to the one-point compactification of the following map
\[
\begin{array}{rcl}
(\rr^n \otimes f)(-) + x  \co  nU & \to & nV \\
(u_1 \oplus \dots \oplus u_n)
& \mapsto &
(f(u_1) \oplus \dots \oplus f(u_n)) + x
\end{array}
\]
Now we consider the other direction around the square.
Start with a pair $(f,x)$ as before, then we may write
$x=x_1 \oplus \dots \oplus x_n$, using the projection of $\rr^n$
onto its $n$ factors of $\rr$. We then see that
\[
nS((f,x))= (S(f,x_1) , \dots , S(f,x_n) ) \in \wcal_n(S^U,S^V)
\]
Applying $\mu_n$ to this element gives the map below.
\[
(u_1 \oplus \dots \oplus u_n) \mapsto
f(u_1) + x_1 \oplus \dots \oplus f(u_n) + x_n
\]
It is now clear that the second square commutes.

We define the vertical adjunctions. The three pairs are defined similarly.
\begin{definition}
  The functors
  \[
  \begin{array}{rcl}
  \alpha_n^* \co \Sigma_n \lca \jcal_1 \Sp &\lra& \interorthdevcat \\
  \alpha_n^* \co O(n) \lca \jcal_1 \Sp &\lra& \orthdevcat \\
  \mu_n^* \co \Sigma_n \lca \wcal \Sp &\lra& \devcat 
  \end{array}
  \]  
  are defined as precomposition with $\alpha_n$ (or $\mu_n$) along with a change of group
  action. For more details, see  
 \cite[Definition 5.1]{barneseldred} and \cite[Definition 8.2]{barnesoman13}.
\end{definition}

The functor $\alpha_n^*$ has a left adjoint,
${\jcal_1 \smashprod_{\jcal_n} -}$. On an object $E$
of $\interorthdevcat$ it is given by the formula
\[
({\jcal_1 \smashprod_{\jcal_n} -} E)(V)
= \int^{U \in \jcal_n} E(U) \smashprod \jcal_1(U,nV).
\]
We let $O(n)$ act on $nV = \rr^n \otimes V$ via
the standard action on $\rr^n$. This induces an action on 
$\jcal_1(U,nV)$ and we let $O(n)$
act diagonally on the smash product
$E(U) \smashprod \jcal_1(U,nV)$. 
Equally we can define $i^* \jcal_1 \smashprod_{i^* \jcal_n} -$.
The functor $\mu_n^*$ has a left adjoint 
$\wcal_n \smashprod_{\wcal} -$  given by a similar formula.

\begin{proposition}\label{prop:alphan}  
  The adjoint pairs $(\wcal_n \smashprod_{\wcal} -, \mu_n^*)$
  and $(i^* \jcal_1 \smashprod_{i^* \jcal_n} -, \alpha_n^*)$
  are Quillen equivalences.
\end{proposition}
\begin{proof}
  The first adjunction is a Quillen pair by  \cite[Proposition 5.4]{barneseldred}.
  The second by a similar argument as in \cite[Section 8]{barnesoman13}.
\end{proof}

\begin{proposition}
  The square labelled 1 in Figure \ref{fig:big} 
  (displayed at the beginning of this section)
  commutes up to natural isomorphism:
  $\alpha_n^* \circ S^* \cong nS^* \circ \mu_n^*$.
  The square of left adjoints also commutes.
\end{proposition}
\begin{proof}
  We have shown above that $S \circ \alpha_n = \mu_n \circ nS$.
  Since the right adjoints are all defined in terms of precomposition
  the result follows. We note that $\alpha_n$ and $\mu_n$ also change
  the way the group $\Sigma_n$ acts, but it is easily checked that
  they do so in compatible ways.
  \end{proof}

\begin{corollary}
  The adjoint pair $(L_{nS}, nS^*)$ is a Quillen equivalence.
\end{corollary}
\begin{proof}
  This is a Quillen pair by Lemma \ref{lem:nsquillen}. The other
  adjunctions in the square are Quillen equivalences,
  see   Lemma \ref{lem:squillen},
  Proposition \ref{prop:alphan} and
  \cite[Proposition 5.4]{barneseldred}.
  Hence this Quillen pair is a Quillen equivalence.
\end{proof}


\begin{lemma}
  The square labelled 2 in Figure \ref{fig:big} (also displayed at the beginning of this section)
  commutes up to natural isomorphism. That is,
  $\alpha_n^* \circ i^* \cong i^* \circ \alpha_n^*$.
  The square of left adjoints also commutes.
\end{lemma}
\begin{proof}
  Let $E \in O(n) \jcal_1 \Sp$. Then $(i^* \alpha_n^* E)(V)
  = i^* E(nV)$, where we have altered the action to involve
  the action of $O(n)$ on $nV$. We also have $(\alpha_n^* i^* E)(V)
  = i^* E(nV)$, with the action altered similarly.
\end{proof}

\subsection{Square 3 commutes}\label{subsec:square3}


Now we can describe the extent to which the square below (labelled 3 in Figure \ref{fig:big})
commutes. Recall that $\res_0^n: = i_n^*$ and $\ind_0^n$ is the $n^{th}$ derivative (Definition \ref{def:nthderiv}).
\[
\xymatrix@C+0.8cm@R+1cm{
\devcat
\ar@<-5pt>[r]_{nS^*}
\ar@<-5pt>[d]_-{(-)/\Sigma_n  \circ \mapdiag^*}
\ar @{} [drr] |{3}
&
\ar@<-5pt>[l]_{L_{nS}}
\interorthdevcat
\ar@<+5pt>[r]^{O(n)_+ \smashprod_{\Sigma_n} (-)}
&
\ar@<+5pt>[l]^-{i^*}
\orthdevcat
\ar@<-5pt>[d]_-{(-)/O(n) \circ \res_0^n }
\\
\ar@<-5pt>[u]_-{\diff_n}
\wcal \Top_{\nhomog}
\ar@<-5pt>[rr]_{S^*}
&&
\ar@<-5pt>[ll]_{L_S}
\ar@<-5pt>[u]_{\ind_0^n}
\jcal_0 \Top_{\nhomog}
}
\]
We claim that the following composite functors agree up to natural isomorphism.
\[
(-)/O(n) \circ \res_0^n \circ O(n)_+ \smashprod_{\Sigma_n} (-) \circ nS^*
\cong
S^* \circ (-)/\Sigma_n \circ \mapdiag^*
\]

\begin{remark}
Unfortunately, these composites consist of both left and right Quillen functors.
So it will not follow automatically that we have a commuting square on the
level of homotopy categories. Thus we will delay homotopical considerations
until after we have proven the claimed commutativity.
\end{remark}

The key fact comes from the commuting squares of diagram categories of
Section \ref{subsec:square1and2}:
\[
\xymatrix@C+1cm{
\jcal_0 \ar[r]^{i_n} \ar[d]^{S}
\ar@/^20pt/[rr]^{\R^n \ox-} &
\jcal_n
\ar[r]^{\alpha_n} \ar[d]^{nS} &
\jcal_1
\ar[d]^{S}\\
\wcal
\ar[r]^{\mapdiag}\ar@/_20pt/[rr]_{(-)^n} &
\wcal_n
\ar[r]^{\mu_n} &
\wcal \\
}
\]
The lefthand square of this diagram tells us that
$\res_0^n \circ  nS^* =  S^* \circ \mapdiag^*$
We also know that for any $\Sigma_n$-space $X$,
\[
(O(n)_+ \smashprod_{\Sigma_n} X)/O(n) \cong X/\Sigma_n
\]
and that this isomorphism is natural in $X$.
Hence, for any $E$ in $\devcat$ there are natural isomorphisms
\[
\begin{array}{rcl}
\left( \res_0^n (O(n)_+ \smashprod_{\Sigma_n} (nS^* E) \right)/O(n)
& = &
\left( (O(n)_+ \smashprod_{\Sigma_n} (E \circ nS \circ i_n) \right)/O(n) \\
& \cong &
\left (E \circ nS \circ i_n \right) /\Sigma_n \\
& \cong &
\left (E \circ \mapdiag \circ S \right) /\Sigma_n \\
& = &
S^* \left ((\mapdiag^* E) /\Sigma_n \right)
\end{array}
\]

Thus, square 3 commutes up to natural isomorphism.

As already noted, $S^*$ and $nS^*$ are right Quillen functors
while $(-)/\Sigma_n$ and $(-)/O(n)$ are left Quillen functors.
We must then look a little deeper to find a homotopically meaningful
description of how these functors (or rather, their derived counterparts) commute.

\begin{proposition}\label{prop:bighocommute}
The following diagram of homotopy categories 
commutes up to natural isomorphism. 
We use $L$ and $R$ to indicate where we have taken left or right derived functors. 
\[
\xymatrix@C+0.8cm@R+1cm{
\ho(\Sigma_n\circlearrowleft \wcal \Sp)
\ar[r]^{R S^*}
\ar[d]_-{R \mu_n^\ast}
&
\ho(\Sigma_n\circlearrowleft \jcal_1 \Sp)
\ar[r]^{O(n)_+ \smashprod_{\Sigma_n}^L (-)}
&
\ho(O(n) \circlearrowleft \jcal_1 \Sp)
\ar[d]^-{R \alpha_n^\ast}
\\
\ho(\devcat)
\ar[d]_-{L (-)/\Sigma_n \circ \mapdiag^*}
&
&
\ho(\orthdevcat)
\ar[d]^-{L (-)/O(n) \circ \res_0^n}
\\
\ho(\wcal \Top_{\nhomog})
\ar[rr]_{R S^*}
&&
\ho(\jcal_0 \Top_{\nhomog})
}
\]
\end{proposition}
\begin{proof}
By \cite[Theorem 6.6]{barneseldred} the composite of 
$R \mu_n^*$ and $L (-)/\Sigma_n \circ \mapdiag^*$
applied to a $\Sigma_n$-spectrum $\Theta$
is weakly equivalent to the formula 
\[
X \mapsto \Omega^\infty \big( (\Theta \smashprod A^{\smashprod n})/h \Sigma_n \big)
\]
of Theorem \ref{thm:classifications}.
Similarly, by \cite[Theorem 10.1]{barnesoman13} 
the composite of 
$R \alpha_n^*$ and $L (-)/O(n) \circ \res_0^n$
applied ot an $O(n)$-spectrum $\Psi$
is weakly equivalent to the formula 
\[
V \mapsto \Omega^\infty \big( (\Psi_E \smashprod S^{\rr^n \otimes V} )/h O(n) \big)
\]
of Theorem \ref{thm:classifications}.
The result then follows by 
Proposition \ref{prop:restrictionhomog}.
\end{proof}

\begin{theorem}
The following square commutes up to natural isomorphism. 
\[
\xymatrix@C+0.6cm@R+1cm{
\ho(\devcat)
\ar[r]^{R nS^*}
\ar[d]_-{L (-)/\Sigma_n  \circ \mapdiag^*}
&
\ho(\interorthdevcat)
\ar[r]^{O(n)_+ \smashprod_{\Sigma_n}^L (-)}
&
\ho(\orthdevcat)
\ar[d]^-{L (-)/O(n) \circ \res_0^n }
\\
\ho(\wcal \Top_{\nhomog})
\ar[rr]_{R S^*}
&&
\ho(\jcal_0 \Top_{\nhomog})
}
\]
\end{theorem}
\begin{proof}
Take some $E \in \devcat$. Since we are interested in the homotopy category, 
this is the same as choosing some spectrum $\Theta$ with a $\Sigma_n$-action.
By Proposition \ref{prop:bighocommute} the two images of 
$\Theta$ in $\ho(\jcal_0 \Top_{\nhomog})$ agree. 
This, combined with the commutativity of Squares 1 and 2 (on the level of homotopy
categories) completes the proof. 
\end{proof}



\addcontentsline{toc}{part}{{\small Bibliography}}

\bigskip

\begin{tabular}{ll}
\begin{tabular}{l}
David Barnes \\
Pure Mathematics Research Centre, \\
Queen's University, \\
Belfast BT7 1NN, UK\\
\url{d.barnes@qub.ac.uk}\\
\end{tabular}
&
\begin{tabular}{l}
Rosona Eldred \\
Mathematisches Institut, \\
Universit\"{a}t M\"{u}nster\\
Einsteinstr. 62, 48149 M\"{unster}, Germany\\
\url{eldred@uni-muenster.de}
\end{tabular}
\end{tabular}

\end{document}